\documentclass[11pt]{amsart}

\usepackage[latin2]{inputenc}
\usepackage[T1]{fontenc}
\usepackage{euscript}
\usepackage{amsmath}
\usepackage{amsthm}
\usepackage{amssymb}
\usepackage[matrix,arrow]{xy}
\usepackage{enumerate}
\usepackage{times}

\usepackage{mathptmx}
\usepackage{helvet}
\usepackage{courier}
\usepackage{graphicx}
\usepackage{multicol}
\usepackage{footmisc}

\usepackage{amscd}
\usepackage{epic}
\usepackage{eufrak}
\usepackage{enumerate}
\usepackage{array}

\numberwithin{equation}{section}

\theoremstyle{plain}
\newtheorem{theorem}[equation]{Theorem}
\newtheorem{lemma}[equation]{Lemma}
\newtheorem{proposition}[equation]{Proposition}
\newtheorem{corollary}[equation]{Corollary}
\newtheorem{conjecture}[equation]{Conjecture}

\theoremstyle{definition}

\newtheorem{example}[equation]{Example}
\newtheorem{remark}[equation]{Remark}

\newtheorem{bekezd}[equation]{}

\newcommand{\cO}{\mathcal{O}}

\newcommand{\cV}{\mathcal{V}}

\newcommand{\cN}{\mathcal{N}}
\newcommand{\cB}{\mathcal{B}}
\newcommand{\cI}{\mathcal{I}}

\newcommand{\setC}{\mathbb{C}}\newcommand{\bC}{\mathbb{C}}
\newcommand{\bQ}{\mathbb{Q}}
\newcommand{\bR}{\mathbb{R}}
\newcommand{\bZ}{\mathbb{Z}}

\newcommand{\bP}{\mathbb{P}}
\def\bH{\mathbb H}

\title{Links of rational singularities, L-spaces and LO fundamental groups}

\author{Andr\'as N\'emethi}
\address{R\'enyi Institute of Mathematics,  1053 Budapest,   Re\'altanoda u. 13--15, Hungary.}
\email{nemethi.andras@renyi.mta.hu}

\thanks{The  author is partially supported by the OTKA Grants 100796 and K112735.}
\date{}

\begin{document}

\maketitle
\pagestyle{myheadings} \markboth{{\normalsize
Andr\'as  N\'emethi}}{ {\normalsize Links of rational singularities, L-spaces and LO fundamental groups}}

\begin{abstract}
We prove that the link of a complex normal surface singularity is an L--space if and only if
the singularity is rational. This via a  result of
Hanselman, J. Rasmussen, S. D. Rasmussen and Watson \cite{HRRW}, proving the conjecture of
Boyer, Gordon and Watson \cite{BGW}, shows that a singularity link is not
rational if and only if
its fundamental group is left--orderable if and only if it admits a 
coorientable taut foliation.
\end{abstract}

\section{Introduction}\label{sec:introduction}

In the present note we wish to connect three areas of mathematics, algebraic geometry
(especially, the theory of local complex normal surface singularities), low dimensional
topology (Heegaard Floer homology and foliations), and
group theory (left--orderable property).
There are well--defined interplays  between them: links of such singularities are oriented
3--manifolds, whose fundamental groups (with minor exceptions) characterize the corresponding
3--manifolds and the topology of the singularity.
We show that certain basic objects (fundamental in
classification procedures in these three rather independent theories) can be identified
in a surprising way. In singularity theory we target the  rational singularities; by definition
they are those germs with vanishing geometric genus. This vanishing (although it is
 analytic in nature) was characterized combinatorially by the plumbing graph of the link
 by Artin and Laufer (graphs satisfying the property are called `rational graphs') \cite{Artin62,Artin66,Laufer72}. In 3--dimensional topology we consider
 the family of L--spaces, introduced by Ozsv\'ath  and Szab\'o, they are characterized by
 the vanishing of the reduced Heegaard Floer homology, and are key fundamental
 objects in recent developments in topology  \cite{OSz4a,OSz4b}. Being a
 rational singularity link,  or an  L--space,  will be compared
 with the left--orderability of the corresponding fundamental groups.

In fact, the link $M$  of a complex normal surface singularity $(X,o)$  is a  special plumbed
3--manifold, plumbed along a connected, negative definite graph. In this note we will be interested only in rational homology sphere 3--manifolds, hence the corresponding
plumbing graphs are trees of $S^2$'s. The connection between singularity theory and topology
imposed  by the link had deep influences in both directions and created several bridges.
One of them is the introduction of the {\it lattice cohomology} $\{\bH^q\}_{q\geq 0}(M)$ of such
3--manifolds by the author \cite{NLC} (see also \cite{NOSZ}). Although
$\bH^*(M)$ is defined combinatorially from the graph, it can be compared with several analytic invariants, e.g with the geometric genus as well. In particular, in \cite{NOSZ,NLC} is proved:
\begin{theorem}\label{th:intr1}
$(X,o)$ is a rational singularity if and only if the reduced lattice cohomology
of its link $M$ satisfies $\bH_{red}^0(M)=0$; or, equivalently,  $\bH_{red}^*(M)=0$.
\end{theorem}
On the other hand, in \cite{NLC} the author formulated the following conjecture
\begin{conjecture}\label{conj:intr1}
The Heegaard Floer homology and the lattice cohomology  of $M$ are isomorphic (up to a shift in degrees):
$$HF^+_{red,even/odd}(-M,\sigma)=\oplus _{q\ even/odd} \ \bH^q_{red}(M,\sigma)[-d(M,\sigma)],$$
where $\sigma\in {\rm Spin}^c(M)$, and $d(M,\sigma)$ are the $d$--invariants of $HF^+(M,\sigma)$.
\end{conjecture}
In particular, the above conjecture predicts that $HF^+_{red}(M)=0$ (that is, $M$ is an L--space)
if and only if $\bH^*_{red}(M)=0$, which is equivalent with the rationality of the graph
 by Theorem \ref{th:intr1}.

The goal of  the present note is to prove the above  prediction:
\begin{theorem}\label{th:intr2}
A singularity link is an L--space if and only if the singularity is rational.
\end{theorem}
In fact, one direction of the statement is already known. The author introduced the notion of
`bad vertices' of a graph \cite{NOSZ,2E}
(for the definition see \ref{ss:graphs}; it is a  generalization of  a similar definition
of Ozsv\'ath and Szab\'o from \cite{OSzP}). In this way, a graph without bad vertices is rational;
a graph with one bad vertex is a graph,  which becomes rational after a
`(negative) surgery at that  vertex'. In
 particular, the number of bad vertices measures  how far the graph is from being  rational.
Related to Conjecture \ref{conj:intr1} in \cite{NLC} is proved:
\begin{theorem}\label{th:intr3}
If the number of bad vertices of the plumbing graph
is $\leq 1$ then Conjecture \ref{conj:intr1} is true.
\end{theorem}
This was generalized in \cite{OSSz} for $\leq 2$ bad vertices.

Since the above theorem applies  for rational links,
Theorems \ref{th:intr1} and \ref{th:intr3} imply that
the link of a rational  singularity  is an L--space.

The opposite direction was obstructed by the lack of characterizations of the L--spaces
(at least in some language, which can be reformulated inside of singularity theory).
This obstruction was broken recently by several results in this direction,
whose final form is the main result of Hanselman, J. Rasmussen, S. D. Rasmussen and Watson \cite{HRRW}:

 \begin{theorem}\label{th:intr4}
 If $M$ is a closed, connected orientable graph manifold then the following are equivalent:

 \ \ (i) $M$ is not an L--space;

 \  (ii) $M$ has left--orderable (LO) fundamental group;

 (iii) $M$ admits a $C^0$ coorientable taut foliation.
 \end{theorem}
 Recall that a group $G$ is left-orderable if there exists a strict total ordering $<$
 of $G$ such that $g<h$ implies $fg<fh$ for all $f,g,h\in G$. (By convention, the trivial group is not LO.)

The equivalence (ii)$\Leftrightarrow$(iii) was established by Boyer and Clay \cite{BC1},
(iii)$\Rightarrow$ (i) by Boyer and Clay in \cite{BC2}. The equivalence (i)$\Leftrightarrow$(ii)
was conjectured by Boyer, Gordon and Watson \cite{BGW}.
The above Theorem \ref{th:intr4} was the final answer to this conjecture.
For the history and partial
contributions see the introduction and references from \cite{HRRW} (and the references therein).
The needed material will be reviewed in \ref{ss:cutting}.

This allows us to reformulate the  remaining  implication of  Theorem \ref{th:intr2} as follows:
if $M$ is the link of a non--rational singularity then $\pi_1(M)$ is LO, hence not an L--space.

In the proof of this statement the following facts will be crucial:

(A) The characterization of the rational graphs via Laufer's algorithm (Laufer's
computation sequence), and also the graph--combinatorics  of bad vertices;

(B) A theorem of Boyer, Rolfsen and Wiest  \cite{BRW}, which states that for compact, irreducible,
$\bP^2$--irreducible 3--manifold $M$, $\pi_1(M)$ is LO if and only if there exists a non--trivial
homomorphism $\pi_1(M)\to L$, where $L$ is any LO group. In particular, since $\bZ^r$ is LO for any
$r\in\bZ_{>0}$, if $H_1(M,\bQ)\not=0$ then using the abelianization map we obtain that
$\pi_1(M)$ is LO.

(C) A theorem of Clay, Lidman and Watson \cite{CLW} regarding the behaviour of LO property
with respect to free product with amalgamation
 (more precisely, with respect to
 the decomposition of $M$ along a torus and closing the pieces along `LO--slopes'), cf. \ref{ss:cutting}.

(D) The equivalences (i)$\Leftrightarrow$ (ii) from Theorem \ref{th:intr3} above
(at the `initial step' combined with Theorem \ref{th:intr2}).

\vspace{2mm}

Hence Theorems \ref{th:intr2} and \ref{th:intr4} combined provide:

 \begin{theorem}\label{th:intr5}
 If $M$ is the link of a normal surface singularity (that is, if $M$ is the plumbed  manifold
 associated with a connected, negative definite graph),  then the following are equivalent:

 \ \ (i) $M$ is the link of a non--rational singularity (i.e., the graph is not a `rational graph');

 \  (ii) $M$ is not an L--space;

   (iii) $M$ has left--orderable (LO) fundamental group;

 (iv) $M$ admits a $C^0$ coorientable taut foliation.
 \end{theorem}

Note that an integral homology sphere $M$ is a rational link if and only if
$M=S^3$ or $M=\Sigma(2,3,5)$, the link of the Brieskorn $E_8$--singularity
$\{x^2+y^3+z^5=0\}$, see e.g.  \cite{Five}.
In particular, if $M$ is an integral homology sphere singularity link,
not of type $S^3$ or $\Sigma(2,3,5)$, then (ii)-(iii)-(iv) above  are automatically satisfied.
(For left-orderability of $\pi_1(M)$, whenever $M$ is a integral homology 3-sphere, irreducible and toroidal graph manifold, see \cite{CLW}.)

 \begin{remark}{\bf (An alternative proof. The existence of taut foliations)}
Although  the equivalence of (i) with (iv) follows from Theorem  \ref{th:intr4}, we note that
an alternative proof of the main theorem can be done by inductive construction of the taut foliation
on the corresponding links (instead of proving the LO property of their fundamental group).
More precisely, instead to use the equivalence (i)$\Leftrightarrow$(ii) of Theorem \ref{th:intr4}
and the results (B) and (C) listed above regarding LO--properties, one can run a very similar
proof,  based on the very same inductive steps.  It  starts with the equivalence (i)$\Leftrightarrow$(iv),
part (C) will be replaced by the gluing procedure of the foliations, while
 part (C) (theorem of Boyer, Rolfsen and Wiest  \cite{BRW}) will be replaced by an
extension of a result of Eisenbud, Hirsch and Neumann \cite{EHN}, which guarantees
the existence of a transversal  cooriented foliation whenever the plumbing graph is negative semidefinite.
(The  Eisenbud, Hirsch and Neumann case is valid
 for Seifert fibered manifolds). The details are given in  \ref{ss:3.3}--\ref{end}.
 Particular cases  already  appeared in   \cite{M}.
\end{remark}

\bekezd {\bf Applications.}  Several results valid from  singularity theory can be reinterpreted
via the above correspondence in terms of L--spaces.
E.g., since rational graphs are stable with respect to taking subgraphs, or decreasing the decorations of the vertices, we obtain:

\begin{corollary}\label{cor:intr1}
Negative definite
plumbing graphs of plumbed L--spaces are stable with respect to taking subgraphs, or decreasing the decorations of the vertices.
\end{corollary}

Using stability with respect to finite coverings we obtain the following 
(for the proof see \ref{proofcor2}). 

\begin{corollary}\label{cor2}
Assume that we have a finite covering $M_1\to M_2$
of graph 3--manifolds associated with
connected negative definite plumbing graphs. The covering is either unbranched,
or it is branched  with branch locus 
$B_2\subset M_2$. In the second case we assume that $M_2$ admits a 
negative definite plumbing representation,
such that all the connected components of $B_2$ are represented by
(generic) $S^1$--fibers of certain JSJ Seifert components of the plumbing.   
 Then $M_2$ is an L--space whenever $M_1$ is an L--space.  
\end{corollary} 
One can find easily 
(even non--branched)  coverings  when $M_2$ is an L--space but $M_1$ is not.

\begin{example} {\bf (Coverings)} \label{cor:intr2}
Let $K\subset S^3$ be an  embedded algebraic link
(the link of an isolated  plane curve singularity). The cyclic $\bZ_n$ covering of
 $S^3$ branched along $K$ is an L--space if and only if

 $\bullet$ \ $n=2$ and $K$ is the link of an A-D-E (simple) plane curve singularity, or

 $\bullet$ \ $n>3$ and $K$ is the torus link $T_{2,m}$ with $\frac{1}{m}+\frac{1}{n}>\frac{1}{2}$.
 \end{example}
For the proof of the statement see \ref{be:cor}.
The short proofs of \ref{cor:intr1}-\ref{cor2}-\ref{cor:intr2}
 can be compared with sometimes long computations involving
Heegaard Floer homology or the arithmetics of foliations.

\begin{example} {\bf (The Seifert fibered case)}  The link of a weighted
homogeneous normal surface singularity is a Seifert 3--manifold. In \cite{pinkham}
Pinkham computed the geometric genus for such singularity  in terms of the
Seifert invariants in the case when  the link is a rational homology sphere. 
The  vanishing of the corresponding expression provides a numerical rationality criterion 
in terms of Seifert invariant. Hence, the main result provides a new criterion for the topological 
 properties (ii)-(iii)-(iv) from Theorem \ref{th:intr5}.
Here is this  new numerical criterion.   

Assume that the star--shaped graph
has $\nu\geq 3$ legs, the central vertex  $v_0$ is decorated by $e_0$, and the $i$-th leg by
$-b_{i1}, \ldots, -b_{is_i}$, where $[b_{i1},\ldots, b_{is_i}]=b_{i1}-\frac{1}{b_{i2}-\cdots}=\alpha_i/\omega_i$
is the (Hirzebruch) continued fraction with $b_{ij}\geq 2$. The positive integers $\{(\alpha_i,\omega_i)\}_{i=1}^\nu$ are the Seifert invariants with $0<\omega_i<\alpha_i$,
${\rm gcd}(\alpha_i,\omega_i)=1$. ($v_0$ is connected to the vertices 
decorated by $-b_{i1}$.) 
We assume that the graph is negative definite, that is,  $e:=e_0+\sum_i\omega_i/\alpha_i<0$.
Then, by \cite{pinkham}, $M$ is {\it non}--rational if and only if 
\begin{equation}\label{eq:1}
\sum_i \lfloor -l \omega_i/\alpha_i \rfloor \leq le_0 -2 \ \ \ \mbox{for at least
 one $l\in \bZ_{\geq 0}$.}
\end{equation}
 This looks very different than the previous criterions used in topology, e.g. for the existence of foliations (results of 
Eisenbud, Hirsch, Jankins, Neumann, Naimi \cite{EHN,JN,Na}, here we follow 
\cite{LM}).  Let us recall it  for $\nu=3$.
 
Following \cite{LM}  we say that $(x,y,z)\in (\bQ\cap (0,1))^3$ is {\it realizable}
if  there exist coprime integers $m>a>0$ such that up to a permutation of $x,y,z $ one has 
$x<a/m$, $ y<(m-a)/m$,  $ z<1/m$. 

Then $M(\Gamma)$ admits a coorientable transversal foliation 
if and only  if one of the following holds:
\begin{equation}\label{eq:2}  \begin{array}{l}
\mbox{  $e_0=-1$ \ and \ $\{\beta_i/\alpha_i\}_{i=1,2,3}$ is realizable;}\\
\mbox{  $e_0=-2$ \ and \ $\{(\alpha_i-\beta_i)/\alpha_i\}_{i=1,2,3}$ is realizable.}
\end{array} \end{equation}
A direct arithmetical proof of the equivalence of these two criterions will be proved 
in another note. 
\end{example}

\section{Preliminaries}\label{s:prel}

\subsection{The combinatorics of a graph.} \label{ss:graphs}

For more details regarding this subsection see \cite{Five}.

Let $\Gamma$ be a connected negative definite plumbing graph. Since both L--spaces and
rational singularity links are rational homology spheres, without loss of generality we can
assume that  our plumbing graphs are trees and all genus decorations are zero.
Let $\cV$ denote the set of vertices,  $e_v$  the  `Euler' decoration of $v\in\cV$, and
 $\delta_v$ the  valency of $v$ in $\Gamma$. The vertex $v$ is called a node if $\delta_v\geq 3$.

We associate with $\Gamma$ its plumbed 4--manifold $P(\Gamma)$ too,  and  its lattice
$L:=H_2(P(\Gamma),\bZ)$ generated freely by $\{E_v\}_{v\in\cV}$ together with the negative
definite  intersection form  $I:=(E_v,E_w)_{v,w}$. Let $K\in L\otimes \bQ$ be the canonical cycle
defined by the adjunction relations $(K+E_v,E_v)+2=0$ for all $v$, and
 set (the `Riemann-Roch expression') $\chi(l)=-(K+l,l)/2$.  Furthermore,
let $Z_{min}$ be the non--zero, unique
 minimal cycle $Z\in L$ with $(Z,E_v)\leq 0$ for all $v\in\cV$ \cite{Artin62,Artin66}.

  If $(X,o)$ is a complex normal surface singularity, embedded in some
   $(\bC^n,o)$, then its link is $M:=X\cap \{|z|=\epsilon\}$ for $0<\epsilon\ll 1$.
   If  $\tilde{X}\to X$ is one of the
  resolutions of $X$, then the geometric genus $p_g(X,o)$ is defined via the sheaf cohomology
  of the structure sheaf $p_g:=h^1(\cO_{\tilde{X}})$.
  Note that any dual graph $\Gamma$
  of a good resolution of $\tilde{X}\to X$ might serve as a plumbing graph for the link,
  and   $\tilde{X}$ is diffeomorphic with $P(\Gamma)$.  Furthermore, any connected
  negative definite graph is the resolution graph of certain singularity \cite{Gr}, and
  plumbed 3--manifolds of such graphs are irreducible \cite{Ne}.

  $(X,o)$ is called rational if $p_g=0$.
  E.g., all quotient singularities are rational. Artin characterized rationality in terms of
  (any) resolution (or, plumbing) graph of $(X,o)$: $(X,o)$ is rational if and only if $\chi(Z_{min})\geq 1$ (in any resolution) \cite{Artin62,Artin66}.
  Graphs with this property will be called rational.

  Laufer in \cite{Laufer72} provided a simple way to detect the rationality of a graph.
  A computation sequence in $L$, ending in $Z_{min}$ is a sequence $l_0=\sum_v E_v,
  \ l_1,\ldots, l_t=Z_{min}$,  $l_i\in L$, such that for any $i$ there exists $v(i)\in \cV$
  such that $l_{i+1}=l_i+E_{v(i)}$. Laufer proposed the following construction of a computation sequence: start with $l_0=\sum_vE_v$, and construct each $l_i$ inductively. If $l_i$ is already
  constructed, and $(l_i,E_v)\leq 0$ for all $v$ then stop, and take $l=t$. If there exists
  $v$ with $(l_i,E_v)>0$ then take such a $v$ as $v(i)$ and set $l_{i+1}=l_i+E_{v(i)}$.
  Then the  procedure stops after finite steps, and the end cycle $l_t$ is
  always $Z_{min}$. By \cite{Laufer72}, $\Gamma$ is rational if and only if
  at each step along the sequence $(l_i,E_{v(i)})=1$. (The choice of $v(i)$, hence of the
   sequence, usually is not unique. Nevertheless,  $l_t=Z_{min}$ in all cases,
    and the rationality criterion
   is also independent of all choices.)

   Using this algorithm, one verifies the following facts:

(a) Subgraphs of rational graphs are rational.

(b) Decreasing the Euler decoration of a rational graph provides a rational graph.

(c) If $e_v\leq - \delta_v$ for all $v$ 
then $\Gamma$ is rational (in fact, $l_0=Z_{min}=\sum_vE_v$).


A set of vertices $\cB\subset \cV$ is called `set of bad vertices' if replacing the decorations
$e_v$, $v\in\cB$, by sufficiently negative integers $e_v'\ll e_v$ we get a rational
graph \cite{NOSZ,2E}.
The set $\cB$ (even with minimality property) is not unique.
If $\cB=\emptyset$ works then obviously $\Gamma$ is rational.

Set $m(\Gamma):=\min_{\cB \ bad}|\cB |$. From the above (a) we obtain
the following.
\begin{lemma}\label{lem:1}
If $\cB\subset \cV(\Gamma)$ is a set of bad vertices of \ $\Gamma$, and \ $\Gamma'$
is a subgraph of \ $\Gamma$, then $\cB':=\cB\cap \cV(\Gamma')$ might serve as  a set of
bad vertices for \ $\Gamma'$. In particular, $m(\Gamma')\leq m(\Gamma)$.
\end{lemma}

\subsection{Cutting $M$ along an incompressible JSJ torus}\label{ss:cutting}

\bekezd\label{be:CLW}
 First we recall the notion of Dehn surgery and a result of Clay, Lidman and Watson \cite{CLW}.

Assume that $N$ is a 3--manifold with incompressible torus boundary.
A  {\it slope} is a  primitive  element of $H_1(\partial M,\bZ)/\{\pm 1\}$. To any slope $\alpha$
one associates the Dehn filling of $N$ along $\alpha $, $N(\alpha)$,
obtained by identifying the boundary of a solid torus $D^2\times S^1$ with
 $\partial N$ in such a way that $\partial D^2\times \{*\}$ is glued to $\alpha$. According to
 \cite{CLW}, a slope $\alpha$ is called left--orderable if $\pi_{1}(M(\alpha))\simeq
 \pi_1(M)/\langle \alpha\rangle$ is LO.

 \begin{theorem}\label{th:CLW} \cite{CLW}
 Suppose that $M_1$ and $M_2$ are 3--manifolds with incompressible torus boundaries, and
 $\phi:\partial M_1\to \partial M_2$ is a homeomorphism such that $M=M_1\cup_{\phi} M_2$
 is irreducible. If there exists a left--orderable slope $\alpha$ of $M_1$ such that
 $\phi_*(\alpha)$ is also a left--orderable   slope, then $\pi_1(M) $ is LO.
 \end{theorem}

\bekezd \label{be:cutting} Let $\Gamma$ be as in \ref{ss:graphs}. We assume that $\Gamma$ is
 minimal, that is, it has no vertex $v$ with $e_v=-1$ and $\delta_v\leq 2$. In that case,
 any edge $e=(v,w)$
 on a (minimal) path connecting two nodes determines an incompressible
 JSJ torus $T\subset M$. Cutting $M$ along $T$ provides two manifolds $M_v$ and $M_w$,
both with  incompressible torus boundaries. We also write $\Gamma\setminus \{e\}=\Gamma_v\sqcup
\Gamma_w$ (with $v\in \cV(\Gamma_v)$);
 $M_v$ is obtained from $M(\Gamma_v)$ by removing a solid torus.

 We wish to apply Theorem \ref{th:CLW} for the decomposition $M=M_v\cup_{\phi}M_w$
 (where $\phi:\partial M_w\to -\partial M_v$), as an inductive step, in order to prove the left-orderability of $\pi_1(M)$. In this case, $\Gamma$ will be a graph with $m(\Gamma)\geq 2$, 
 and we search for a slope $\alpha\in H_1(\partial M_w,\bZ)/\{\pm 1\}$ with the next properties:

 (i) $M_v(\phi_*(\alpha))$ is still non-rational connected negative definite
 graph manifold (together with $M$ it satisfies  certain inductive step, being `less complicated'
 than $M$);

 (ii) $M_w(\alpha)$ has first  Betti number one, hence its fundamental group is LO.

 Condition (i) will be guaranteed by the convenient choice of the edge $e$.
 Next, we describe $\alpha$ in terms of $\Gamma_w$ in order to have condition (ii) satisfied.

 \bekezd \label{be:det} Assume that $\Gamma$ is a not--necessarily negative definite,  or
not--necessarily  connected graph, without loops,
 when we allow even rational decorations as well. Let $I$ be its intersection form, and set
 $\det(\Gamma):=\det(-I)$. Obviously, if $\Gamma$ (that is, $I$) is negative definite,
 then $\det (\Gamma  )>0$. Linear algebra and Sylvester's criterion show  the following:
 \begin{lemma}\label{lem:det}
  (a)  Fix an edge $e=(v,w)$, and let $[v,w]$ be the subgraph with vertices $v,w$ and edge $e$.
  Then $\det(\Gamma)=\det(\Gamma\setminus e)-\det (\Gamma\setminus [v,w])$.

  (b) $\Gamma$ is negative definite if and only if $\Gamma\setminus e$ is negative definite and
  $\det(\Gamma)>0$.
\end{lemma}

\bekezd\label{be:alpha} (Continuation of \ref{be:cutting}.)
In a suitable basis of $H_1(\partial M_w,\bZ)$, given by  the plumbing construction,
the slope $\alpha$ can be represented by a rational number $r=r(\alpha)$.
Furthermore, since the gluing in the plumbing construction (in these bases) is the matrix
{\tiny  $\begin{pmatrix} 0& 1\\1 &0 \end{pmatrix}$}, $\phi_*(\alpha)$ corresponds to the
rational number $1/r$.
Hence, $M_v(\phi_*(\alpha))$ and
$M_w(\alpha)$ can be represented by the graphs $\Gamma_v(1/r)$ and $\Gamma_w(r)$ below

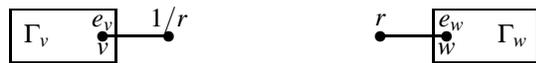
\begin{picture}(181,34)(-60,3)
\drawline(20,10)(60,10)(60,30)(20,30)(20,10)	
\drawline(180,10)(220,10)(220,30)(180,30)(180,10)
\put(80,20){\circle*{4}}
\put(160,20){\circle*{4}}
\put(55,20){\circle*{4}}
\put(185,20){\circle*{4}}
\drawline(55,20)(80,20)
\drawline(160,20)(185,20)
\put(30,20){\makebox(0,0){\small{$\Gamma_v$}}}
\put(210,20){\makebox(0,0){\small{$\Gamma_w$}}}
\put(55,15){\makebox(0,0){\small{$v$}}}
\put(185,15){\makebox(0,0){\small{$w$}}}
\put(55,25){\makebox(0,0){\small{$e_v$}}}
\put(187,25){\makebox(0,0){\small{$e_w$}}}
\put(80,26){\makebox(0,0){\small{$1/r$}}}
\put(160,26){\makebox(0,0){\small{$r$}}}
\end{picture}

Here the new vertices decorated by rational numbers can be replaced by strings
whose decorations are given by the entries of the corresponding continued fraction
(write $r$ as $[e_1,\ldots, e_s]=e_1-\frac{1}{e_2-\cdots}$, with $e_1\leq -1$,
$e_i\leq -2$ for $i\geq 2$,
 then the $\{e_i\}$'s are the decorations of the string).

 \begin{lemma}\label{lem:2}
 Take $r:=-\det(\Gamma_w\setminus w)/\det(\Gamma_w)$. Then
 $\det(\Gamma_w(r))=0$, hence $\Gamma_w(r)$ is negative semidefinite. Furthermore,
   $\Gamma_v(1/r)$ is negative definite with $\det(\Gamma_v(1/r))=\det(\Gamma)/
   \det(\Gamma_w\setminus w)$.
 \end{lemma}
\begin{proof}
Apply Lemma \ref{lem:det} for the new edges (and for $e$ in $\Gamma$).
\end{proof}

\section{Proofs}\label{s:proof}

\bekezd\label{be:1} {\bf Proof of Theorem \ref{th:intr2}.}
We show that if $\Gamma$ is not a rational graph then $M(\Gamma)$ is not an L--space.
We  consider the family of graphs $\Gamma$  with $m(\Gamma)\geq 1$, and run
induction over  the number of nodes $\cN$ of graphs.
For any $\Gamma$ the inductive construction provides another graph with less nodes,
  however, we have to be careful not to obtain
   one with $m(\Gamma)=0$. On the other hand, having
$m(\Gamma)\leq |\cN(\Gamma)|$ (since the set of nodes might serve as a set of bad
   vertices), by increasing the number of nodes we necessarily arrive to the situation when
   $m(\Gamma)=1$; these are exactly the starting cases of the induction.

\vspace{2mm}

\noindent {\bf Starting case.} \ Hence,
to start the induction, consider any connected, negative definite graph with
$m(\Gamma)=1$ (and with arbitrarily many nodes). Then $M(\Gamma) $ is not an L--space by
Theorem \ref{th:intr3}, and $\pi_1(M(\Gamma))$ is left-orderable by Theorem \ref{th:intr4}.

\vspace{2mm}

Next, assume that the statement is proved for any $\Gamma$ with $N-1$ nodes (where $N\geq 2$),
and take a graph $\Gamma$ with $N$ nodes and $m(\Gamma)\geq 2$. We wish to show that
$\pi_1(M(\Gamma))$ is LO. We can also assume that $\Gamma$ is minimal.

  \vspace{2mm}

\noindent {\bf Case 1.} \ Assume that $\Gamma$ has a vertex $v$ such that at least two of the
connected components of $\Gamma\setminus v$ contain nodes of $\Gamma$. Let $\{\Gamma_i\}_{i\in \cI}$ be the connected components of $\Gamma\setminus v$,  $\Gamma^\downarrow$ be the graph obtained from
$\Gamma$ after replacing $e_v$ with $e_v'\ll e_v$ (in such a way that the multiplicity of $E_v$ in
$Z_{min}(\Gamma^\downarrow)$ is 1), $\Gamma_i(v)^\downarrow$ is the minimal connected subgraph
of  $\Gamma^\downarrow$ supported on $\Gamma_i\cup v$ (hence $v$ is glued by an edge to $\Gamma_i$
and decorated by $e_v'$).
Since $m(\Gamma)\geq 2$ we get that $\Gamma^\downarrow$ is not rational.

When we run the Laufer algorithm
for $Z_{min}(\Gamma^\downarrow)$, the different parts $\Gamma_i(v)^\downarrow$ do not interfere, and  $Z_{min}(\Gamma^\downarrow)|_{\Gamma_i(v)^\downarrow}=Z_{min}
(\Gamma_i(v) ^\downarrow)$. Since in the computation sequence of $Z_{min}(\Gamma^\downarrow)$
there is at least one `jump' with $(l_i,E_{v(i)})\geq 2$, such a jump will apear for some $i\in\cI$
in the computations sequence of  $ Z_{min}(\Gamma_i(v) ^\downarrow)$ too. In particular,
this graph $\Gamma_i(v) ^\downarrow$ is not rational. Choose an index $j\in\cI$, $j\not=i$,
such that $\Gamma_j$ contains at least one node of $\Gamma$ 
(the assumption of case 1 is used here).

 Then choose the edge $e=(v,w)$ in such a way that $\Gamma_w$ (as in \ref{be:cutting})
 equals $\Gamma_j$. Then we apply Theorem \ref{th:CLW} for $\Gamma_v(1/r)$ and
 $\Gamma_w(r)$, where $r$ is as in Lemma \ref{lem:2}. Since $\det(\Gamma_w(r))=0$,
 $\pi_1(M_w(\alpha))$ is LO by the theorem of Boyer, Rolfsen and Wiest  \cite{BRW}
 reviewed  in the introduction. Next,  $\Gamma_v(1/r)$ is a connected,
  negative definite graph, and we claim that it is not rational. Indeed, the non-rational graph
 $\Gamma_i(v)^\downarrow$ is obtained from it by decreasing the support and decreasing the decorations. Since $\Gamma_v(1/r)$ contains less nodes than $\Gamma$, by the inductive step
 $\pi_1(M_v(\phi_*(\alpha)))$ is LO. Hence $\pi_1(M(\Gamma)))$ is LO by Theorem \ref{th:CLW}.

  \vspace{2mm}

  \noindent {\bf Case 2.}\ Assume that such a vertex $v$ (as in Case 1) does not exist.
  This can happen only if $\Gamma$ has exactly two nodes $(n_1,n_2)$
  and they are adjacent ($|\cN|\leq 1$ would imply $m(\Gamma)\leq 1$).
  Let us blow up the edge $(n_1,n_2)$ and  denote the new graph by $\Gamma'$, and the newly created $(-1)$--vertex by $v$. If $v$ is  the only bad vertex of $\tilde{\Gamma}$
  (that is, with similar  notation as above, $\tilde{\Gamma}^\downarrow$ is rational)
  then we conclude as in the
  starting case by Theorems \ref{th:intr3} and \ref{th:intr4}.

 If $\tilde{\Gamma}^\downarrow$ is not rational then we repeat the arguments of Case 1 for this graph and vertex $v$.

 This ends the proof of Theorem \ref{th:intr2}.

\bekezd \label{ss:3.3} {\bf An alternative proof of Theorem \ref{th:intr2} showing
the existence of the foliation.}

In this subsection we indicate how the above proof from \ref{be:1} should
be changed if we wish to replace the LO property of the fundamental groups with the existence
of the foliations.

We will run the same induction guided by the same geometric construction (`decomposition' of 
$M$ into $M_v(\phi_*(\alpha))$ and $M_w(\alpha)$ with graphs $\Gamma_v(1/r)$ and $\Gamma_w(r)$).

As a part of the preparation, we have to find the analogue of part (B) from the introduction
(that is, of Theorem of Boyer, Rolfsen and Wiest \cite{BRW}, 
 which guarantees that $\pi_1(M_w(\alpha))$ is LO).
The  next fact  was already considered in \cite{M}.

\begin{proposition}\label{prop:1}
Let $M_w(\alpha)$ be the manifold constructed in \ref{be:alpha}. Then it admits a transversal 
coorientable foliation.
\end{proposition}

Here, by transversality  of the foliation we mean that it is transverse to the Seifert 
fibration of each JSJ component. In particular, they are taut. (In \cite{B} 
is also shown that they are $\bR$--covered with a non--trivial action on $Homeo^+(\bR)$.)

The statement of proposition \ref{prop:1} follows by induction over the number of JSJ 
Seifert components of $M_w(\alpha)$. If
 $\Gamma_w(r)$ contains only one node, that is,  $M_w(\alpha)$ is Seifert, then 
the statement is covered exactly by a theorem of Eisenbud, Hirsch and Neumann \cite{EHN}.

If $\Gamma_w (r)$ has more nodes then let us consider a `separating edge' (on a path connecting 
two nodes)  $e'=(v',w')$ in $\Gamma_w(r)$. Then we repeat the cutting construction 
from \ref{be:cutting}, applied now to $\Gamma_w(r)$, and along $e'$.
The graph $\Gamma_w(r)\setminus e'$ has two components, $\Gamma_w(r)_{v'}$ and 
$\Gamma_w(r)_{w'}$.
 Since $\Gamma_w(r)$ is negative semidefinite, the determinants of the proper
 subgraphs $\Gamma_w(r)_{w'}\setminus w'$ and $\Gamma_w(r)_{w'}$ are positive. 
 Set $r':=-\det (\Gamma_w(r)_{w'}\setminus w')/\det(\Gamma_w(r)_{w'})$, and construct the new graphs $\Gamma_w(r)_{v'}(1/r')$ and $\Gamma_w(r)_{w'}(r')$:

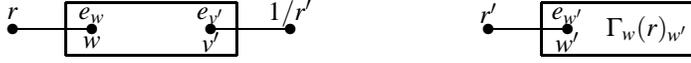
 \begin{picture}(181,34)(-60,3)
\drawline(0,10)(65,10)(65,30)(0,30)(0,10)	
\drawline(180,10)(240,10)(240,30)(180,30)(180,10)
\put(85,20){\circle*{4}}\put(10,20){\circle*{4}}
\put(160,20){\circle*{4}}
\put(55,20){\circle*{4}}\put(-20,20){\circle*{4}}
\put(190,20){\circle*{4}}
\drawline(55,20)(85,20)\drawline(-20,20)(10,20)
\drawline(160,20)(190,20)
\put(220,20){\makebox(0,0){\small{$\Gamma_w(r)_{w'}$}}}
\put(55,15){\makebox(0,0){\small{$v'$}}}\put(10,15){\makebox(0,0){\small{$w$}}}
\put(190,15){\makebox(0,0){\small{$w'$}}}
\put(55,25){\makebox(0,0){\small{$e_{v'}$}}}
\put(190,25){\makebox(0,0){\small{$e_{w'}$}}}\put(10,25){\makebox(0,0){\small{$e_w$}}}
\put(85,26){\makebox(0,0){\small{$1/r'$}}}
\put(160,26){\makebox(0,0){\small{$r'$}}}\put(-20,26){\makebox(0,0){\small{$r$}}}
\end{picture}

Similarly as in Lemma \ref{lem:2}, 
 $\det(\Gamma_w(r)_{w'}(r'))=0$.  
 Since $\Gamma_w(r)_{w'}$ is negative definite,
  $\Gamma_w(r)_{w'}(r')$ is negative semidefinite.
  Furthermore,
   $\det(\Gamma_w(r)_{v'}(1/r'))=\det(\Gamma_w(r))/\det (\Gamma_w(r)_{w'}\setminus w')=0$ too.
 Since $\Gamma_w(r)_{v'}$ is negative definite, we obtain that $\Gamma_w(r)_{v'}(1/r')$
 is negative semidefinite too. 
   
By induction, both manifolds $M_{w}(\alpha)_{v'}(\phi'_*(\alpha'))$ and
$M_{w}(\alpha)_{w'}(\alpha')$ admit transversal  foliations. Let us restrict them
to the manifolds $M_{w}(\alpha)_{v'}$ and
$M_{w}(\alpha)_{w'}$, both with torus boundary. 
The foliation of $M_{w}(\alpha)_{w'}(\alpha')$ induces a foliation of the boundary torus.
Since each leaf of the foliation of the torus extends to a foliation of the solid torus
used in the Dehn filling, the leaf is homological trivial in the solid torus, hence it 
consists of (certain parallel) copies of the slope. Hence the induced foliation on the torus consists of parallel simple/primitive loops parallel to the slope. The same is valid for the 
manifold   $M_{w}(\alpha)_{v'}(\phi'_*(\alpha'))$ and for the newly created torus boundary. 
Since, by construction, they are represented by $r'$ and $1/r'$ respectively, 
by  plumbing  they can be glued together into a foliation of $M_w(\alpha)$.

Since the gluing respects the JSJ decompostion, and the two pieces (by induction) are 
transversal, the newly created foliation is transversal  too.
Moreover,  in \cite[Lemma 5.5]{BRW}  is proved that transversal  foliations on Seifert pieces 
(with base space $S^2$) are cooriantable. This property will be preserved under the above gluing
as well. 

\bekezd\label{end}
 Now we are ready to prove the following statement:
if $M$ is the manifold associated with a non--rational (connected, negative definite)
graph then it admits a transversal  coorintable foliation, hence it is not an L--space
(via 
Theorem \ref{th:intr4}). In the proof we use Theorem \ref{th:intr3}, the above 
Proposition \ref{prop:1}, and the equivalence (i)$\Leftrightarrow(iii)$ from Theorem \ref{th:intr4}.

We will run the very same induction. In the {\bf starting case}, when $m(\Gamma)=1$, 
$M(\Gamma)$ is not an L--space by Theorem \ref{th:intr3}, hence it admits the foliation by 
Theorem \ref{th:intr4}. 

In {\bf Case 1}, by inductive step both $M_w(\alpha)$ and $M_v(\phi_*(\alpha))$ admit 
foliations, whose restrictions  to $M_w$ and $M_v$ can be glued similarly as in the proof of 
Proposition \ref{prop:1} in order to get a foliation of $M(\Gamma)$. 
Next, {\bf Case 2} is a combination of these two parts,
similarly as in \ref{be:1}.

\bekezd \label{proofcor2} {\bf Proof of Corollary \ref{cor2}.}
The assumption quarantees that 
$(M_2,B_2)$ has an
analytic  realization as link of $(X,D,o)$, where $(X,o)$ is a normal surface singularity and 
$(D,o)$ a (Weil) divisor on it.
Fix such an analytic realization.
 Then, by a theorem of Stein \cite{Stein,GrRe1}, the monodromy representation of the (regular, topological) covering over $M_2\setminus B_2$, and the analytic structure of $(X,o)$ determines a (unique) normal surface singularity $(Y,o)$, and  a finite map
$(Y,o)\to (X,o)$ branched along $(D,o)$, which at the level of links  induces $M_1\to M_2$.   
Then use the fact that $(X,o)$ is rational whenever $(Y,o)$ is rational, see e.g. 
\cite[Proposition 5.13]{K}.

\bekezd \label{be:cor} {\bf Proof of Example \ref{cor:intr2}.} Let $f(x,y)$ be the equation of
the plane curve singularity with link $K$. Then the $\bZ_n$ cyclic cover $M$ is the link of the
hypersurface singularity  $\{f(x,y)=z^n\}\subset (\bC^3,o)$. Hence $M$ is an L--space if and only if $\{f(x,y)=z^n\}$ is rational. A hypersurface singularity is rational if and only if it is
of type A-D-E \cite{Five}. If $n=2$ then $\{f(x,y)=z^2\}$ is simple if and only if
$f$ itself is simple. Assume next $n>2$. Since the multiplicity of a simple surface singularity
is 2, the multiplicity of $f$ should be 2. In particular, by splitting lemma, $f$ in some local coordinates has the form $x^2+y^m$. But the Brieskorn singularity $x^2+y^m=z^n$ is rational
if and only if $1/2+1/m+1/n>1$.

\end{document}